\documentclass[12pt,reqno]{amsart}

\usepackage{amsmath,amssymb,amsthm, geometry}

\usepackage[colorlinks=false, urlcolor=blue, final]{hyperref}

\newtheorem{thm}{Theorem}[section]
\newtheorem{lem}[thm]{Lemma}
\newtheorem{cor}[thm]{Corollary}
\newtheorem{prop}[thm]{Proposition}

\theoremstyle{definition}

\theoremstyle{remark}
\newtheorem{ex}[thm]{Example}

\allowdisplaybreaks

\DeclareMathOperator{\spn}{span}
\DeclareMathOperator{\Prim}{Prim}
\DeclareMathOperator{\KP}{KP}

\newcommand{\F}{\mathbb{F}}

\newcommand{\C}{\mathbb{C}}
\newcommand{\Z}{\mathbb{Z}}
\newcommand{\N}{\mathbb{N}}
\newcommand{\inv}{^{-1}}

\begin{document}

\title[The center of  a Kumjian-Pask algebra]
{\boldmath{The center of a Kumjian-Pask algebra}}

\author[J. H. Brown]{Jonathan H. Brown}
\address{Department of Mathematics\\Kansas State University\\
138 Cardwell Hall\\
Manhattan, KS 66506-2602\\
USA.}
\email{ 	brownjh@math.ksu.edu}
\author[A. an Huef]{Astrid an Huef}
\address{Department of Mathematics and Statistics\\
University of Otago\\
Dunedin 9054\\
New Zealand.}
\email{astrid@maths.otago.ac.nz}

\date{12 September, 2012}

\begin{abstract} The Kumjian-Pask algebras are path algebras associated to higher-rank graphs, and generalize the Leavitt path algebras.  We study the center of simple Kumjian-Pask algebras and characterize commutative Kumjian-Pask algebras.
\end{abstract}

\maketitle


\section{Introduction}

Let $E$ be a directed graph $E$ and let $\F$ be a field.  The Leavitt path algebras $L_\F(E)$ of $E$  over $\F$ were first introduced in \cite{AA} and \cite{AMP}, and have been widely studied since then. Many of the properties of Leavitt path algebras can be inferred from properties of the graph,  and for this reason  provide a convenient way to construct examples of algebras with a particular set of attributes.  The Leavitt path algebras are the algebraic analogues of the graph $C^*$-algebras associated to $E$.  In \cite{Tom11}, Tomforde  constructed an analogous Leavitt path algebra $L_R(E)$ over a commutative ring $R$ with $1$, and introduced more techniques from the graph $C^*$-algebra setting to study it.  

In \cite{A-PCaHR11}, Aranda Pino, Clark, an Huef and Raeburn generalized Tomforde's construction and associated to a higher-rank graph $\Lambda$  a graded algebra $\KP_R(\Lambda)$ called the \emph{Kumjian-Pask algebra}. Example~7.1 of \cite{A-PCaHR11} shows that even the class of commutative Kumjian-Pask algebras over a field is strictly larger than the class  of Leavitt path algebras over that field.

In this paper we study the center of Kumjian-Pask algebras. In  \S\ref{motivation} we work over $\C$ and show how the embedding of $\KP_\C(\Lambda)$ in the $C^*$-algebra $C^*(\Lambda)$ can be used together with the Dauns-Hofmann theorem to deduce that the center of a simple Kumjian-Pask algebra is either $\{0\}$ or isomorphic to $\C$. 

More generally, it follows from Theorem~\ref{thm: ZKP fin}, that the center of a ``basically simple'' (see page~\pageref{page-basic}) Kumjian-Pask algebra $\KP_R(\Lambda)$ is either zero or is isomorphic to the underlying ring $R$. Thus our  Theorem~\ref{thm: ZKP fin} generalizes the analogous theorem  for Leavitt path algebras over a field \cite[Theorem~4.2]{A-PC11}, but our proof techniques are very different and more informative. Indeed, the Kumjian-Pask algebra is basically simple  if and only if the graph $\Lambda$ is cofinal and aperiodic, and our proofs show explicitly which of these graph properties are needed to infer various properties of elements in the center.

In Proposition~\ref{prop com} we show that a Kumjian-Pask algebra of a $k$-graph $\Lambda$ is commutative if and only it is  a direct sum of rings of Laurant polynomials in $k$-indeterminates,   if  and only if $\Lambda$  is a disjoint union of copies of the category $\N^k$.  This generalizes Proposition~2.7 of \cite{A-PC11}.


\section{Preliminaries}
We view $\N^k$ as a category with one object and composition given by addition.
 We call a countable category $\Lambda=(\Lambda^0, \Lambda, r, s)$ a $k$-graph if there exists a functor $d:\Lambda\to \N^k$ with the \emph{unique factorization property}: for all $\lambda\in \Lambda$, $d(\lambda)=m+n$ implies there exist unique $\mu,\nu\in \Lambda$ such that $d(\mu)=m, d(\nu)=n$ and $\lambda=\mu\nu$.  Using the unique factorization property, we identify the set of objects $\Lambda^0$  with the set of morphisms of degree $0$. Then, for $n\in \N^k$, we set $\Lambda^n:=d\inv(n)$, and call $\Lambda^n$ the paths of shape $n$ in $\Lambda$ and $\Lambda^0$ the vertices of $\Lambda$. 
A path $\lambda\in \Lambda$ is \emph{closed}  if $r(\lambda)=s(\lambda)$.

For $V,W\subset \Lambda^0$, we set $V\Lambda:=\{\lambda\in \Lambda: r(\lambda)\in V\}$, $\Lambda W:=\{\lambda\in \Lambda: s(\lambda)\in W\}$ and $V\Lambda W:= V\Lambda\cap \Lambda W$; the sets $V\Lambda^n, \Lambda^n W$ and $V\Lambda^n W$ are defined similarly. For simplicity we write $v\Lambda$ for $\{v\} \Lambda$. 

A $k$-graph $\Lambda$ is \emph{row-finite} if $|v\Lambda^n|<\infty$ for all $v\in \Lambda^0$ and $n\in \N^k$ and has \emph{no sources} if $v\Lambda^n\neq \emptyset$ for all $v\in\Lambda^0$ and $n\in \N^k$.  We assume throughout that $\Lambda$ is a row-finite $k$-graph with no sources.

Following \cite[Lemma~3.2(iv)]{RS07}, we say that a $k$-graph $\Lambda$ is  \emph{aperiodic}
if for every $v\in \Lambda^0$ and $m\neq n\in {\mathbb N}^k$ there exists $\lambda\in v\Lambda$ such that
$d(\lambda)\geq m\vee n$ and 
\begin{equation*}\lambda(m,m+d(\lambda)-(m\vee n))\neq \lambda(n,n+d(\lambda)-(m\vee n)).
\end{equation*} This formulation of aperiodicity is equivalent to the original one from \cite[Definition~4.3]{KumPas00} when $\Lambda$ is a row-finite graph with no sources, but is often more convenient since it only involves finite paths.

Let $\Omega_k:=\{(m,n)\in \N^k: m\leq n\}$.  As in \cite[Definition~2.1]{KumPas00} we define an \emph{infinite path} in $\Lambda$ to be a degree-preserving functor $x:\Omega_k\to \Lambda$, and denote the set of infinite paths by $\Lambda^\infty$.   As in \cite[Definition~4.1]{KumPas00} we say $\Lambda$ is \emph{cofinal} if for every infinite path $x$ and every vertex $v$ there exists  $m\in \N^k$  such that $v\Lambda x(m)\neq\emptyset$.

For each $\lambda\in \Lambda$ we introduce a \emph{ghost path} $\lambda^*$; for $v\in \Lambda^0$ we take $v^*=v$. We write $G(\Lambda)$ for the set of  ghost paths and $G(\Lambda^{\neq 0})$ if we exclude the vertices.

Let $R$ be a commutative ring with $1$. Following \cite[Definition~3.1]{A-PCaHR11},
a \emph{Kumjian-Pask $\Lambda$-family} $(P,S)$ in an $R$-algebra $A$  consists of  functions $P:\Lambda^0\to A$ and $S:\Lambda^{\neq 0}\cup G(\Lambda^{\neq 0})\to A$ such that
\begin{enumerate}
\item[(KP1)]  $\{P_v:v\in \Lambda^0\}$ is a set of  mutually orthogonal idempotents;  
\item[(KP2)]  for $\lambda, \mu\in \Lambda^{\neq 0}$ with $r(\mu)=s(\lambda)$,
\[S_{\lambda}S_\mu=S_{\lambda\mu},\  S_{\mu^*}S_{\lambda^*}=S_{(\lambda\mu)^*},\  P_{r(\lambda)}S_\lambda=S_\lambda=S_\lambda P_{s(\lambda)},\  P_{s(\lambda)}S_{\lambda^*}=S_{\lambda^*}=S_{\lambda^*}P_{r(\lambda)};\]
\item[(KP3)] for all $\lambda,\mu\in \Lambda^{\neq 0}$ with $d(\lambda)=d(\mu)$, we have
$S_{\lambda^*}S_\mu=\delta_{\lambda,\mu}P_{s(\lambda)}$;
\item[(KP4)] for all $v\in \Lambda^0$ and $n\in \N^k\setminus\{0\}$, we have
$P_v=\sum_{\lambda\in v\Lambda^n} S_\lambda S_{\lambda^*}$.
\end{enumerate}
By  \cite[Theorem~3.4]{A-PCaHR11} there is 
an $R$-algebra $\KP_R(\Lambda)$, generated by a nonzero Kumjian-Pask $\Lambda$-family $(p, s)$, with the following universal property:  whenever $(Q,T)$ is a Kumjian-Pask $\Lambda$-family in an $R$-algebra $A$, then there is a
unique $R$-algebra homomorphism $\pi_{Q,T}:\KP_R(\Lambda)\to A$ such that
\begin{equation*}\label{defpiqt}
\pi_{Q,T}(p_v) = Q_v,\  \pi_{Q,T}(s_{\lambda}) = T_{\lambda} \text{\ and \ } \pi_{Q,T}(s_{\mu^*}) =T_{\mu^*}\text{\ for $v\in \Lambda^0$ and $\lambda,\mu\in\Lambda^{\neq 0}$.}
\end{equation*}
Also by Theorem~3.4 of \cite{A-PCaHR11}, the subgroups 
\[
\KP_R(\Lambda)_n:=\spn_R\{s_\lambda s_{\mu^*}:\lambda,\mu\in \Lambda\text{\ and\ }d(\lambda)-d(\mu)=n\}\quad\quad (n\in\Z^k)
\]
give a $\Z^k$-grading of $\KP_R(\Lambda)$.  Let $S$ be a $\Z^k$-graded ring; then by the graded-uniqueness theorem \cite[Theorem~4.1]{A-PCaHR11}, a graded homomorphism  $\pi:\KP_R(\Lambda)\to S$ such that $\pi(rp_v)\neq 0$ for nonzero $r\in R$ is injective.

We will often write elements $a\in \KP_R(\Lambda)\setminus\{0\}$ in the \emph{normal form}  of \cite[Lemma~4.2]{A-PCaHR11}: there exists $m\in \N^k$ and a finite $F\subset \Lambda\times \Lambda^m$ such that $a=\sum_{(\alpha,\beta)\in F} r_{\alpha, \beta} s_\alpha s_{\beta^*}$ where $r_{\alpha,\beta}\in R\setminus\{0\}$ and $s(\alpha)=s(\beta)$.


\section{Motivation}\label{motivation}

When $A$ is a simple $C^*$-algebra (over $\C$, of course), it follows from the Dauns-Hofmann Theorem (see, for example, \cite[Theorem~A.34]{tfb}) that the center $Z(A)$ of $A$ is isomorphic to $\C$ if $A$ has an identity and is $\{0\}$ otherwise.  Let $\Lambda$ be a row-finite $k$-graph without sources.   In this short section we  deduce that the center of a simple Kumjian-Pask algebra $\KP_\C(\Lambda)$ is either isomorphic to $\C$ or is $\{0\}$.

\begin{lem}
\label{lem c*}
Suppose $A$ is a simple $C^*$-algebra. If $A$ has an identity, then $z\mapsto z1_{A}$ is an isomorphism of $\C$ onto  the center $Z(A)$ of $A$. If $A$ has no identity, then  $Z(A)=\{0\}$.
\end{lem}

\begin{proof} Since $A$ is simple, $\Prim A =\{\star\}$, and $f\mapsto f(\star)$ is an isomorphism of $C_b(\Prim A)$ onto $\C$. By the Dauns-Hofmann Theorem, $C_b(\Prim A)$ is isomorphic to the center $Z(M(A))$ of the multiplier algebra $M(A)$ of $A$. Putting the two isomorphisms together gives an isomorphism $z\mapsto z1_{M(A)}$ of $\C$ onto $Z(M(A))$.

Now suppose that $A$ has an identity. Then $M(A)=A$, and it follows from the first paragraph that $Z(A)$ is isomorphic to $\C$.

Next suppose that $A$ does not have an identity.  Let $a\in Z(A)$, and let $u_\lambda$ be an approximate identity in $A$ and $m\in M(A)$.  Then $ma=\lim (mu_\lambda) a=a\lim (mu_\lambda)=am$. Thus $Z(A)\subset Z(M(A))$.  Now $Z(A)\subset Z(M(A))\cap A=\{z 1_{M(A)}:z\in \C\}\cap A=\{0\}$.
\end{proof}

\begin{lem}\label{lem-shor}
Let $D$ be  a dense subalgebra of a $C^*$-algebra $A$. Then $Z(A)\cap D= Z(D)$. 
\end{lem}
\begin{proof} Trivially,  $Z(A)\cap D\subset Z(D)$. 
 To see the reverse inclusion,  let $a\in Z(D)$.  Let $b\in A$ and choose $\{d_\lambda\}\subset D$ such that $d_\lambda\to b$. Then $ba=\lim_\lambda d_\lambda a=\lim_\lambda a d_\lambda=ab$. Now  $a\in Z(A)\cap Z(D)\subset Z(A)\cap D$, and hence $Z(A)\cap D= Z(D)$.
\end{proof}

By \cite[Theorem~6.1]{A-PCaHR11}, $\KP_\C(\Lambda)$ is simple if and only if $\Lambda$ is cofinal and aperiodic, so in the next corollary $\KP_\C(\Lambda)$ is simple. Also, $\KP_\C(\Lambda)$ has an identity if and only if $\Lambda^0$ is finite (see Lemma~\ref{lem: KP unital}  below).
\begin{cor}
\label{cor complex} Suppose that $\Lambda$ is a row-finite,  cofinal, aperiodic $k$-graph with no sources. If $\Lambda^{0}$ is finite, then $z\mapsto z1_{\KP_\C(\Lambda)}$ is an isomorphism of $\C$ onto the center $Z(\KP_\C(\Lambda))$ of $\KP_\C(\Lambda)$. If $\Lambda^{0}$ is infinite, then  $\KP_\C(\Lambda)=\{0\}$.
\end{cor}
\begin{proof} Let $(p,s)$ be a generating Kumjian-Pask $\Lambda$-family for $\KP_\C(\Lambda)$ and $(q,t)$ a generating Cuntz-Krieger $\Lambda$-family for $C^*(\Lambda)$.  Then $(q,t)$ is a Kumjian-Pask $\Lambda$-family in $C^*(\Lambda)$, and the universal property of $\KP_\C(\Lambda)$ gives a  $*$-homomorphism $\pi_{q,t}:\KP_\C(\Lambda)\to C^*(\Lambda)$ which takes $s_\mu s_{\nu^*}$ to $t_\mu t_\nu^*$. It follows from the graded-uniqueness theorem that $\pi_{q,t}$ is a $*$-isomorphism onto a dense $*$-subalgebra of $C^*(\Lambda)$ (see Proposition~7.3 of \cite{A-PCaHR11}).  Since $\Lambda$ is aperiodic and cofinal, $C^*(\Lambda)$ is simple by \cite[Theorem~3.1]{RS07}.

Now suppose that $\Lambda^0$ is finite. Then $\KP_\C(\Lambda)$ has identity $1_{\KP_\C(\Lambda)}=\sum_{v\in\Lambda^0} p_v$  and $C^*(\Lambda)$ has identity $1_{C^*(\Lambda)}=\sum_{v\in\Lambda^0} q_v$, and $\pi_{q,t}$ is unital. By Lemma~\ref{lem c*},  $Z(C^*(\Lambda))=\{z1_{C^*(\Lambda)}: z\in\C\}$.  By Lemma~\ref{lem-shor}, $Z(\pi_{q,t}(\KP_\C(\Lambda)))=Z(C^*(\Lambda))\cap \pi_{q,t}(\KP_\C(\Lambda))=\{z1_{C^*(\Lambda)}: z\in\C\}$. Since $\pi_{q,t}$ is unital, $Z(\KP_\C(\Lambda))$ is isomorphic to $\C$ as claimed.

Next suppose that $\Lambda^0$ is infinite. Then $Z(C^*(\Lambda))=\{0\}$ and $Z(\pi_{q,t}(\KP_\C(\Lambda)))=Z(C^*(\Lambda))\cap \pi_{q,t}(\KP_\C(\Lambda))=\{0\}$, giving $\KP_\C(\Lambda)=\{0\}$.
\end{proof}


\section{The center of a  Kumjian-Pask algebra}
Our goal is to extend Corollary~\ref{cor complex} to  Kumjian-Pask algebras over arbitrary rings. Throughout $R$ is a commutative ring with $1$ and $\Lambda$ is a row-finite $k$-graph with no sources.

We will need Lemma~\ref{lem: lin inde} several times. For notational convenience, for $v\in \Lambda^0$,  $s_v$ or $s_{v^*}$ means $p_v$.  So  when $m=0$ Lemma~\ref{lem: lin inde}, says that the set $\{s_{\alpha}: \alpha\in\Lambda\}$ is linearly independent. 
 
\begin{lem}\label{lem: lin inde} Let $m\in \N^k$. Then    $\{s_{\alpha}s_{\beta^*}: \text{$s(\alpha)=s(\beta)$ and  $d(\beta)=m$}\}$ is a  linearly independent subset of $\KP_R(\Lambda)$.\end{lem}

\begin{proof} Let $F$ be a finite subset of $\{(\alpha,\beta)\in \Lambda\times \Lambda^m : s(\alpha)=s(\beta) \}$, and suppose that $\sum_{(\alpha,\beta)\in F}r_{\alpha,\beta}s_{\alpha}s_{\beta^*}=0$.  Fix $(\sigma,\tau)\in F$. Since  all the $\beta$ have degree $m$, using (KP3) twice we obtain
\begin{align}
0&=s_{\sigma^*}\Big( \sum_{(\alpha,\beta)\in F}r_{\alpha,\beta}s_{\alpha}s_{\beta^*}\Big) s_\tau\notag
\\
&=r_{\sigma,\tau}p_{s(\sigma)}+\sum_{(\alpha,\beta)\in F\setminus\{(\sigma,\tau)\}}r_{\alpha,\tau} s_{\sigma^*}s_\alpha s_{\beta^*}s_\tau\notag\\
&=r_{\sigma,\tau}p_{s(\sigma)}+\sum_{\substack{(\alpha,\tau)\in F\\ \alpha\neq \sigma}}r_{\alpha,\tau} s_{\sigma^*}s_\alpha.\label{eq-independence}
\end{align} 
If $d(\sigma)=d(\alpha)$ and $\sigma\neq \alpha$, then $s_{\sigma^*}s_\alpha=0$ by (KP3). If $d(\sigma)\neq d(\alpha)$ then, by \cite[Lemma~3.1]{A-PCaHR11},  $s_{\sigma^*}s_\alpha$ is a linear combination of $s_\mu s_{\nu^*}$ where $d(\mu)-d(\nu)=d(\alpha)-d(\sigma)$. It follows that the $0$-graded component of \eqref{eq-independence} is $r_{\sigma,\tau}p_{s(\sigma)}$. Thus $0=r_{\sigma,\tau}p_{s(\sigma)}$.  But $p_{s(\sigma)}\neq 0$ by Theorem~3.4 of \cite{A-PCaHR11}. Hence $r_{\sigma,\tau}=0$. Since $(\sigma, \tau)\in F$ was arbitrary, it follows that $\{s_{\alpha}s_{\beta^*}: \text{$s(\alpha)=s(\beta)$ and  $d(\beta)=m$}\}$ is linearly independent.
\end{proof}

The next lemma describes properties of elements in the center of $\KP_R(\Lambda)$.

\begin{lem}\label{lem: ZKP arb} Let  $a\in Z(\KP_R(\Lambda))\setminus\{0\}$ be in normal form $\sum_{(\alpha,\beta)\in F} r_{\alpha,\beta} s_{\alpha}s_{\beta^*}$.  
\begin{enumerate}
\item \label{cond: ZKP arb r=r} If $(\sigma,\tau)\in F$, then $r(\sigma)=r(\tau)$.
\item \label{cond: ZKP arb supsat} Let $W=\{v\in\Lambda^0: ~\exists (\alpha,\beta)\in F\text{~~with~~}v=r(\beta) \}.$  If $\mu\in \Lambda W$, then $r(\mu)\in W$.
\item \label{cond: ZKP arb r=s}  If $(\sigma,\tau)\in F$, then there exists $(\alpha,\beta)\in F$ such that $r(\alpha)=r(\beta)=s(\sigma)=s(\tau).$
\item \label{cond: ZKP arb closed path}  There exists $l\in \N\setminus\{0\}$ and $\{(\alpha_i,\beta_i)\}_{i=1}^l\subset F$ such that $\beta_1\cdots \beta_l$ is a closed path in $\Lambda.$
\end{enumerate}
\end{lem}

\begin{proof}  \eqref{cond: ZKP arb r=r} Let $(\sigma,\tau)\in F$.  By  \cite[Lemma~2.3]{BH} we have  $0\neq s_{\sigma^*}as_\tau$. Since $a\in Z(\KP_R(\Lambda))$
\[
0\neq s_{\sigma^*}p_{r(\sigma)}ap_{r(\tau)}s_\tau=s_{\sigma^*}ap_{r(\sigma)}p_{r(\tau)}s_\tau=\delta_{r(\sigma),r(\tau)}s_{\sigma^*}as_\tau.\]
Hence $r(\sigma)=r(\tau)$.

\eqref{cond: ZKP arb supsat} By way of contradiction, assume  there exists $\mu\in \Lambda W$ such that $r(\mu)\notin W$. Then $p_v p_{r(\mu)}=0$ for all $v\in W$.  Thus 
\[
ap_{r(\mu)}= \sum_{(\alpha,\beta)\in F} r_{\alpha,\beta} s_{\alpha}s_{\beta^*}p_{r(\beta)}p_{r(\mu)}=0.
\]
Since $a\in Z(\KP_R(\Lambda))$  we get $s_\mu a=as_\mu=ap_{r(\mu)}s_\mu=0$. 
Since $s(\mu)\in W$, there exist $(\alpha',\beta')\in F$ with $r(\beta')=s(\mu)$.  Then $r(\alpha')=s(\mu)$ also by \eqref{cond: ZKP arb r=r}.  Thus $S:=\{(\alpha,\beta)\in F: s(\mu)=r(\alpha)\}$ is non-empty, and 
\begin{align*}
0=s_\mu a=\sum_{(\alpha,\beta)\in S} r_{\alpha,\beta} s_{\mu\alpha}s_{\beta^*}. 
\end{align*}
But $\{s_{\mu\alpha}s_{\beta^*}: (\alpha,\beta)\in S\}$ is linearly independent by Lemma~\ref{lem: lin inde}, and hence  $r_{\alpha,\beta}=0$ for all $(\alpha,\beta)\in S$. This contradicts the given normal form for $a$.

\eqref{cond: ZKP arb r=s} Let $(\sigma,\tau)\in F$.  Then $s(\sigma)=s(\tau)$ by definition of normal form. By Lemma~2.3 in \cite{BH} we have $s_{\sigma^*} a s_\tau\neq 0$.
Since $a\in Z(\KP_R(\Lambda))$,
\begin{equation}\label{eq com}0\neq s_{\sigma^*} a s_\tau=as_{\sigma^*}  s_\tau=\sum_{(\alpha,\beta)\in F} r_{\alpha,\beta} s_{\alpha}s_{\beta^*}s_{\sigma^*}s_\tau=\sum_{\substack{(\alpha,\beta)\in F\\r(\beta)=s(\sigma)}} r_{\alpha,\beta} s_{\alpha}s_{(\sigma\beta)^*}s_\tau.
\end{equation}
In particular, the set $\{(\alpha,\beta)\in F: r(\beta)=s(\sigma)\}$ is nonempty.  So there exists $(\alpha',\beta')\in F$ such that $r(\beta')=s(\sigma)$. Since  $r(\alpha')=r(\beta')$ from \eqref{cond: ZKP arb r=r}, we are done.

\eqref{cond: ZKP arb closed path} Let $M=|F|+1$. Using \eqref{cond: ZKP arb r=s} there exists a path $\beta_1\dots\beta_M$ such that, for $1\leq i\leq M$, there exists $\alpha_i\in\Lambda$ with $(\alpha_i,\beta_i)\in F$. Since $M>|F|$, there exists $i<j\in\{1,\dots,M\}$ such that $\beta_i=\beta_j$.  Then $\beta_i\dots\beta_{j-1}$ is a closed path.
\end{proof}

The next corollary follows from Lemma~\ref{lem: ZKP arb}\eqref{cond: ZKP arb closed path}.
\begin{cor}
Let $\Lambda$ be a row-finite $k$-graph with no sources and $R$ a commutative ring with $1$. If $\Lambda$ has no closed paths then the center $Z(\KP_R(\Lambda))=\{0\}$. 
\end{cor}

 The next lemma provides a description of elements of the center of $\KP_R(\Lambda)$ when $\Lambda$ is cofinal.  

\begin{lem}\label{lem: ZKP cof}  Suppose that $\Lambda$ is cofinal. Let $a\in Z(\KP_R(\Lambda))\setminus\{0\}$ be in normal form  $\sum_{(\alpha,\beta)\in F} r_{\alpha,\beta} s_{\alpha}s_{\beta^*}$.   Then $\{v\in\Lambda^0: \exists (\alpha,\beta)\in F\text{~~with~~}  v=r(\beta)\}=\Lambda^0$.\end{lem}

\begin{proof} Write $W:=\{v\in\Lambda^0: \exists (\alpha,\beta)\in F\text{~~with~~}  v=r(\beta)\}$.
 Since $\sum_{(\alpha,\beta)\in F} r_{\alpha,\beta} s_{\alpha}s_{\beta^*}$ is in normal form, $F\subset \Lambda\times \Lambda^m$ for some $m\in\N^k$.  Let $n=m\vee (1,1,\ldots, 1)$.  By (KP4), for each $(\alpha,\beta)\in F$,  we have $s_{\alpha}s_{\beta^*}=\sum_{\mu\in s(\alpha)\Lambda^{n-m}} s_{(\alpha\mu)}s_{(\beta\mu)^*}$.  By ``reshaping'' each pair of paths in $F$ in this way, collecting like terms and dropping those with zero coefficients, we see that there exists $G\subset \Lambda\times  \Lambda^n$ and $r'_{\gamma,\eta}\in R\setminus\{0\}$ such that $a=\sum_{(\gamma,\eta)\in G} r'_{\gamma,\eta} s_\gamma s_{\eta^*}$ is also in normal form. By construction, $W'=\{v\in\Lambda^0: \exists (\gamma,\eta)\in G\text{~~with~~}  v=r(\eta)\}\subset W$.

Let $v\in \Lambda^0$.  Using Lemma~\ref{lem: ZKP arb}\eqref{cond: ZKP arb closed path}, there exists $\{(\gamma_i,\eta_i)\}_{i=1}^l\subset G$ such that $\eta_1\cdots \eta_l$ is a closed path.  Since $d(\eta_i)\geq (1,1,\ldots ,1)$ for all $i$,  $x:=\eta_1\cdots\eta_l\eta_1\cdots\eta_l\eta_1\cdots$ is an infinite path. By cofinality, there exist $q\in\N^k$ and $\nu\in v\Lambda x(q)$. By the definition of $x$, there exist  $q'\geq q$ and  $j$ such that $x(q')=r(\eta_j)$. Let $\lambda=x(q,q')$.  Then $\nu\lambda\in v\Lambda W'$. By Lemma~\ref{lem:  ZKP arb}\eqref{cond: ZKP arb supsat}, $v=r(\nu\lambda)\in W'$ as well. Thus $W'=\Lambda^0$, and since $W'\subset W$ we  have $W=\Lambda^0$.
\end{proof}

The next lemma provides a description of elements of the center of $\KP_R(\Lambda)$ when $\Lambda$ is aperiodic.

\begin{lem}\label{lem: ZKP aper} Suppose that $\Lambda$ is  aperiodic and $a\in Z(\KP_R(\Lambda))\setminus\{0\}$.  Then there exist $n\in \N^k$  and $G\subset \Lambda^n$ such that $a=\sum_{\alpha\in G} r_{\alpha} s_{\alpha}s_{\alpha^*}$ is in normal form.\end{lem}

\begin{proof}  Suppose $a\in Z(\KP_R(\Lambda))\setminus\{0\}$ with $a=\sum_{(\alpha,\beta)\in F} r_{\alpha,\beta} s_{\alpha} s_{\beta^*}$ in normal form.   Let $(\sigma,\tau)\in F.$  From Lemma~2.3 in \cite{BH} we know that 
$s_{\sigma^*} a s_\tau\neq 0$.
Let $m=\vee_{(\alpha,\beta)\in F}(d(\alpha)\vee d(\beta)).$ Since $\Lambda$ is aperiodic, by \cite[Lemma~6.2]{HRSW}, there exist  $\lambda\in s(\sigma)\Lambda$ with $d(\lambda)\geq m$ such that 
\begin{equation}\label{eq aper}
\left.\begin{array}{l}\alpha,\beta\in \Lambda s(\sigma),~d(\alpha),d(\beta)\leq m\\
\text{and}~~\alpha\lambda(0,d(\lambda))=\beta\lambda(0,d(\lambda))\end{array}\right\}
\implies \alpha=\beta.
\end{equation}  
The same argument as in \cite[Proposition~4.9]{A-PCaHR11} now shows that $s_{\lambda^*}s_{\sigma^*}as_{\tau}s_{\lambda}\neq 0$.  Since $a\in Z(\KP_R(\Lambda))$,  $0\neq s_{\lambda^*}s_{\sigma^*}as_{\tau}s_{\lambda}=as_{\lambda^*}s_{\sigma^*}s_{\tau}s_{\lambda}=as_{(\sigma\lambda)^*}s_{\tau\lambda}$. Thus
\begin{align*}
0\neq s_{(\sigma\lambda)^*}s_{\tau\lambda}
&=s_{\sigma\lambda(d(\lambda), d(\lambda)+d(\sigma))^*}s_{\sigma\lambda(0,d(\lambda))^*}s_{\tau\lambda(0,d(\lambda))}s_{\tau\lambda(d(\lambda),d(\lambda)+d(\tau))}\\
&=\delta_{\sigma,\tau} s_{\sigma\lambda(d(\lambda), d(\lambda)+d(\sigma))^*}s_{\tau\lambda(d(\lambda),d(\lambda)+d(\tau))}\quad\text{using \eqref{eq aper}}.
\end{align*} Thus $\sigma=\tau$.

Since $(\sigma,\tau)\in F$ was arbitrary we have  $\alpha=\beta$ for all $(\alpha,\beta)\in F$.  Let $G=\{\alpha\in \Lambda: (\alpha,\alpha)\in F\}$ and write $r_{\alpha}$ for  $r_{\alpha,\alpha}$. Note $G\subset\Lambda^n$ because $F\subset \Lambda\times\Lambda^n$ for some. Thus
$a=\sum_{\alpha\in G} r_{\alpha}s_{\alpha}s_{\alpha^*}$ in normal form as desired.  \end{proof}

Our main theorem (Theorem~\ref{thm: ZKP fin}) has two cases: $\Lambda^0$ finite and infinite.

\begin{lem}\label{lem: KP unital}  $\KP_R(\Lambda)$ has an identity if and only if $\Lambda^0$ is finite.\end{lem}

\begin{proof}  If $\Lambda^0$ is finite, then $\sum_{v\in \Lambda^0} p_v$ is an identity for $\KP_R(\Lambda)$.  Conversely, assume that  $\KP_R(\Lambda)$ has an identity $1_{\KP_R(\Lambda)}$.  By way of contradiction, suppose that $\Lambda^0$ is infinite. 
Write $1_{\KP_R(\Lambda)}$ in normal form
$\sum_{(\alpha,\beta)\in F} r_{\alpha,\beta}s_{\alpha}s_{\beta^*}$.
Since $F$ is finite, so is $W:=\{v\in\Lambda^0:\exists (\alpha,\beta)\in F\text{~~with~~}v=r(\beta)\}$.  Thus there exists $w\in \Lambda^0\setminus W$.  Now $p_w=1_{\KP_R(\Lambda)}p_w=\sum_{(\alpha,\beta)\in F} r_{\alpha,\beta}s_{\alpha}s_{\beta^*}p_w=0$ because $w\neq r(\beta)$ for any of the $\beta$.
This contradiction shows that  $\Lambda^0$ must be finite.
\end{proof}

\begin{thm}\label{thm: ZKP fin}  Let $\Lambda$ be a row-finite $k$-graph with no sources and $R$ a commutative ring with $1$. 
\begin{enumerate}
\item\label{lem: ZKP fin 1} Suppose $\Lambda$ is  aperiodic and cofinal,  and that $\Lambda^0$  is finite.  Then $Z(\KP_R(\Lambda))=R1_{\KP_R(\Lambda)}$.
\item\label{lem: ZKP fin 2}  Suppose that $\Lambda$ is cofinal and  $\Lambda^0$ is infinite. Then $Z(\KP_R(\Lambda))=\{0\}$.
\end{enumerate}
\end{thm}

\begin{proof} \eqref{lem: ZKP fin 1} Suppose $\Lambda$ is  aperiodic and cofinal,  and that $\Lambda^0$  is finite. Let $a\in Z(\KP_R(\Lambda))\setminus\{0\}$.  Since $\Lambda$ is aperiodic, by Lemma~\ref{lem: ZKP aper}, there exist $G\subset \Lambda^n$ such that $a=\sum_{\alpha\in G} r_{\alpha} s_{\alpha}s_{\alpha^*}$ is in normal form.  Since $\Lambda$ is row-finite and $\Lambda^0$ is finite, $\Lambda^n$ is finite.  

We claim that $G=\Lambda^n$.  By way of contradiction, suppose that $G\neq \Lambda^n$, and let $\lambda\in \Lambda^n\setminus G.$  Then $as_\lambda=0$ by (KP3).  But since $a\in Z(\KP_R(\Lambda))$,
\[
0=as_\lambda=s_\lambda a=\sum_{\substack{\alpha\in G\\r(\alpha)=s(\lambda)}} r_\alpha s_{\lambda\alpha}s_{\alpha^*}.
\]
Since $\Lambda$ is confinal,  $\{r(\alpha):\alpha\in G\}=\Lambda^0$ by Lemma~\ref{lem: ZKP cof}.  Thus  $S=\{\alpha\in G: r(\alpha)=s(\lambda)\}\neq \emptyset$.  But $\{s_{\lambda\alpha}s_{\alpha^*}:\alpha\in S\}$ is linearly independent by Lemma~\ref{lem: lin inde}. Thus $r_{\alpha}=0$ for $\alpha\in S$, contradicting our choice of $\{r_{\alpha}\}$.  It follows that  $G=\Lambda^n$ as claimed,  and that
\[a=\sum_{\alpha\in \Lambda^n} r_{\alpha} s_{\alpha}s_{\alpha^*}.\]

Next we claim that $r_\mu=r_\nu$ for all $\mu, \nu\in \Lambda^n$.  Let $\mu,\nu\in \Lambda^n$. Let $x\in s(\mu)\Lambda^\infty$.  Since $\Lambda$ is cofinal,  there exists $m\in \N^k$ and $\gamma\in s(\nu)\Lambda s(x(m))$.  Set $\eta=x(0,m)$.  Now
\begin{align*}
r_\mu s_{\nu\gamma} s_{(\mu\eta)^*} 
&=r_\mu s_{\nu\gamma} s_{\eta^*} s_{\mu^*}=s_{\nu\gamma} s_{\eta^*}\sum_{\alpha\in\Lambda^n} r_\mu s_{\mu^*}s_\alpha s_{\alpha^*}\\
&=s_{\nu\gamma} s_{\eta^*}s_{\mu^*}\sum_{\alpha\in\Lambda^n} r_\alpha s_\alpha s_{\alpha^*}
=s_{\nu\gamma} s_{(\mu\eta)^*}a\\
&=as_{\nu\gamma} s_{(\mu\eta)^*}
=\sum_{\alpha\in \Lambda^n} r_\alpha s_\alpha s_{\alpha^*}s_{\nu}s_{\gamma} s_{(\mu\eta)^*}
=r_\nu s_{\nu\gamma} s_{(\mu\eta)^*}.
\end{align*}
Since $s_{\nu\gamma} s_{(\mu\eta)^*}\neq 0$ this implies $r_\mu=r_\nu$. Let $r=r_\mu$.  Now \[a=\sum_{\alpha\in \Lambda^n} rs_\alpha s_{\alpha^*}=\sum_{v\in \Lambda^0} \sum_{\alpha\in v\Lambda^n} r s_\alpha s_{\alpha^*}=r\sum_{v\in \Lambda^0} p_v=r1_{\KP_R(\Lambda)}\]
as desired.  

\eqref{lem: ZKP fin 2}   Suppose there exists $a\in Z(\KP_R(\Lambda))\setminus\{0\}$.  Write $a=\sum_{(\alpha,\beta)\in F} r_{\alpha,\beta} s_{\alpha}s_{\beta^*}$ in normal form.   Then Lemma~\ref{lem: ZKP cof} gives that $\{v\in\Lambda^0: \exists (\alpha,\beta)\in F\text{~~with~~}  v=r(\beta)\}=\Lambda^0$, contradicting that $F$ is finite.
\end{proof}

Simplicity of $C^*(\Lambda)$ played an important role in \S\ref{motivation}. To reconcile this with Theorem~\ref{thm: ZKP fin}, recall from \cite{Tom11} that an ideal $I\in \KP_R(\Lambda)$ is \emph{basic} if $rp_v\in I$ for $r\in R\setminus \{0\}$ then $p_v\in I$, and  that $\KP_R(\Lambda)$ is \emph{basically simple}\label{page-basic} if its only basic ideals are $\{0\}$ and $\KP_R(\Lambda)$. By \cite[Theorem~5.14]{A-PCaHR11}, $\KP_R(\Lambda)$ is basically simple if and only if $\Lambda$ is cofinal and aperiodic (and by \cite[Theorem~6.1]{A-PCaHR11}, $\KP_R(\Lambda)$ is  simple if and only if $R$ is a field and $\Lambda$ is cofinal and aperiodic). Thus Theorem~\ref{thm: ZKP fin} is in the spirit of  Corollary~\ref{cor complex}.

\section{Commutative Kumjian-Pask algebras}

We view $\N^k$ as a category with one object $\star$ and composition given by addition, and use $\{e_i\}_{i=1}^k$ to denote the standard basis of $\N^k$.

\begin{ex}
\label{ex: Bk}
   Let $d:\N^k\to \N^k$ be the identity map. Then $(\N^k, d)$ is a $k$-graph. By \cite[Example~7.1]{A-PCaHR11},     $\KP_R(\N^k)$ is commutative with identity $p_\star$,  and $\KP_R(\N^k)$ is isomorphic to the ring of  Laurent polynomials  $R[x_1, x_1\inv, \ldots, x_k, x_k\inv]$ in $k$ commuting indeterminates.
\end{ex} 

\begin{lem}
\label{lem: dir sum}
Suppose $\Lambda=\Lambda_1\bigsqcup \Lambda_2$ is a disjoint union of two $k$-graphs.  Then $\KP_R(\Lambda)=\KP_R(\Lambda_1)\oplus \KP_R(\Lambda_2)$.
\end{lem}

\begin{proof}
For each $i=1,2$, let $(q^i, t^i)$ be the generating Kumjian-Pask  $\Lambda_i$-family of $\KP_R(\Lambda_i)$,  and let $(p, s)$ be the generating Kumjian-Pask  $\Lambda$-family of $\KP_R(\Lambda)$. Restricting $(p,s)$ to $\Lambda_i$ gives a $\Lambda_i$-family  in $\KP_R(\Lambda)$, and hence the universal property for $\KP_R(\Lambda_i)$ gives a homomorphism $\pi_{p,s}^i:\KP_R(\Lambda_i)\to \KP_R(\Lambda)$ such that $\pi_{p,s}^i\circ (q^i, t^i)=(p,s)$. Each $\pi_{p,s}^i$ is graded, and the graded uniqueness theorem (\cite[Theorem~4.1]{A-PCaHR11}) implies that $\pi_{p,s}^i$  is injective.  

We now identify $\KP_R(\Lambda_i)$ with its image in $\KP_R(\Lambda)$.  If $\mu\in \Lambda_1$ and $\lambda\in \Lambda_2$, then  $s_\mu s_\lambda=s_\mu p_{s(\mu)} p_{r(\lambda)} s_\lambda =0$.  Similarly $s_\lambda s_\mu, s_{\mu^*} s_{\lambda^*}, s_{\lambda^*} s_{\mu^*}, s_\lambda s_{\mu^*}, s_{\mu^*} s_\lambda, s_{\lambda^*} s_\mu, s_\mu s_{\lambda^*}$ are all zero.  Thus $\KP_R(\Lambda_1)\KP_R(\Lambda_2)=0=\KP_R(\Lambda_2)\KP_R(\Lambda_1)$, and the internal direct sum $\KP_R(\Lambda_1)\oplus \KP_R(\Lambda_2)$ is a subalgebra of $\KP_R(\Lambda)$. Finally, $\KP_R(\Lambda_1)\oplus \KP_R(\Lambda_2)$ is  all of $\KP_R(\Lambda)$ since the former contains all the generators of later.  This gives the result.
\end{proof}

\begin{prop}
\label{prop com}
Let $\Lambda$ be a row-finite $k$-graph with no sources and $R$ a commutative ring with $1$.  Then the following conditions are equivalent:
\begin{enumerate}
\item\label{it 1}$\KP_R(\Lambda)$ is commutative;
\item\label{it 2} $r=s$ on $\Lambda$ and $r|_{\Lambda^n}$ is injective;
\item\label{it 3} $\Lambda\cong\bigsqcup_{v\in \Lambda^0} \N^k$;
\item\label{it 4} $\KP_R(\Lambda)\cong \bigoplus_{v\in \Lambda^0} R[x_1, x_1\inv,\ldots, x_k, x_k\inv]$.
\end{enumerate}
\end{prop}

\begin{proof}
$\eqref{it 1}\Rightarrow \eqref{it 2}$. Suppose that $\KP_R(\Lambda)$ is commutative.  By way of contradiction, suppose there exists $\lambda\in \Lambda$ such that $s(\lambda)\neq r(\lambda)$.  
Then $s_{\lambda^*}s_\lambda=s_\lambda s_{\lambda^*}$, and  \[p_{s(\lambda)}=p_{s(\lambda)}^2=p_{s(\lambda)}s_{\lambda^*} s_{\lambda}=p_{s(\lambda)}s_\lambda s_{\lambda^*}=p_{s(\lambda)}p_{r(\lambda)}s_\lambda s_{\lambda^*}=0.\]  But $p_v\neq 0$ for all $v\in\Lambda^0$ by \cite[Theorem~3.4]{A-PCaHR11}. This contradiction gives  $r=s$.  

Next, supppose $\lambda,\mu\in \Lambda^n$ with $\lambda\neq \mu$. By way of contradiction, suppose that $r(\lambda)=r(\mu)$. Since $r=s$, $r(\lambda)=s(\lambda)=s(\mu)=r(\mu)$.  Then 
\[
s_\lambda=p_{r(\lambda)}s_\lambda=p_{s(\lambda)} s_\lambda =p_{s(\mu)}s_\lambda=s_{\mu^*}s_\mu s_\lambda=s_{\mu^*} s_\lambda s_\mu =0
\]
by (KP3).  Now $p_{s(\lambda)}=0$, contradicting that $p_v\neq 0$ for all $v\in\Lambda^0$ by \cite[Theorem~3.4]{A-PCaHR11}. Thus $r$ is injective on $\Lambda^n$.

$\eqref{it 2}\Rightarrow \eqref{it 3}$ Assume that  $r=s$ on $\Lambda$ and that  $r|_{\Lambda^n}$ is injective. Since $r=s$,  $\{v\Lambda v\}_{v\in \Lambda^0}$ is a partition of $\Lambda$.  Since $r$ is injective on $\Lambda^{e_i}$, the subgraph $v\Lambda^{e_i} v$ has a single vertex $v$ and single edge $f^v_i$.  Thus $f_i^v\mapsto e_i$ defines a graph isomorphism $v\Lambda v\to \N^k$.  Hence $\Lambda=\bigsqcup_{v\in \Lambda^0} v\Lambda v\cong \bigsqcup_{v\in \Lambda^0} \N^k$.

$\eqref{it 3}\Rightarrow \eqref{it 4}$ Assume that $\Lambda\cong\bigsqcup_{v\in \Lambda^0} \N^k$. By Lemma~\ref{lem: dir sum}, $\KP_R(\Lambda)$ is isomorphic to $\bigoplus \KP_R(\N^k)$, and by Example~\ref{ex: Bk} each $\KP_R(\N^k)$ is isomorphic to $R[x_1,x_1\inv,\ldots, x_k, x_k\inv]$.  

$\eqref{it 4}\Rightarrow \eqref{it 1}$ Follows since $\bigoplus R[x_1,x_1\inv,\ldots, x_k, x_k\inv]$ is commutative.
\end{proof}




\providecommand{\bysame}{\leavevmode\hbox to3em{\hrulefill}\thinspace}
\providecommand{\MR}{\relax\ifhmode\unskip\space\fi MR }
\providecommand{\MRhref}[2]{%
  \href{http://www.ams.org/mathscinet-getitem?mr=#1}{#2}
}
\providecommand{\href}[2]{#2}

\end{document}